\documentclass{amsproc}

\usepackage{amsmath}
\usepackage{amssymb}
\usepackage{amsthm}
\usepackage[all,2cell]{xy}

\newtheorem{thm}{Theorem}[section]
\newtheorem{lemma}[thm]{Lemma}

\newtheorem{prop}[thm]{Proposition}

\newtheorem{notn}[thm]{Notation}

\newtheorem{rem}[thm]{Remark}
\newtheorem{rmk}[thm]{Remark}
\newtheorem{set}[thm]{Setup}

\def\P{{\mathbb P}}

\def\Q{{\mathbb Q}}
\def\R{{\mathbb R}}
\def\Z{{\mathbb Z}}

\def\cH{{\mathcal{H}}}

\def\cO{{\mathcal{O}}}

\def\Q{{\mathbb{Q}}}
\def\G{{\mathbb{G}}}

\def\f{\varphi}

\def\operatorname#1{\mathop{\rm #1}\nolimits}

\def\Hilb{\operatorname{Hilb}}
\def\Hom{\operatorname{Hom}}

\def\Pic{\operatorname{Pic}}

\def\rank{\operatorname{rank}}

\def\deg{\operatorname{deg}}

\def\det{\operatorname{det}}
\def\coker{\operatorname{coker}}

\begin{document}

\title{Varieties swept out by grassmannians of lines}

\author{Roberto Mu\~noz}
\address{Departamento de Matem\'atica Aplicada, ESCET, Universidad
Rey Juan Carlos, 28933-M\'ostoles, Madrid, Spain}
\email{roberto.munoz@urjc.es; luis.sola@urjc.es}
\thanks{Partially supported by the Spanish government project MTM2006-04785.}

\author{Luis E. Sol\'a Conde}


\dedicatory{Dedicated to Andrew J. Sommese in his 60th birthday.}


\begin{abstract} We classify complex projective varieties
of dimension $2r \geq 8$ swept out by a family of codimension two
grassmannians of lines $\mathbb{G}(1,r)$. They are either 
fibrations onto normal surfaces such that the general fibers are
isomorphic to $\G(1,r)$ or the grassmannian
$\mathbb{G}(1,r+1)$. The cases $r=2$ and $r=3$ are also considered
in the more general context of varieties swept out by codimension
two linear spaces or quadrics.
\end{abstract}

\maketitle

\section{Introduction}\label{sec:intro}

Let $X$ be a smooth complex projective variety. It is a classical
question in algebraic geometry to understand to which extent the
geometry of $X$ is determined by a particular family of subvarieties
of $X$.

Perhaps the first subvarieties that algebraic geometers have
considered in that sense are lines in projective varieties
$X\subset\P^N$. Examples of the use of this idea can be found all
throughout the literature, evolving into the study of rational curves
in algebraic varieties that has become a central part of algebraic
geometry since Mori's landmark work in 1980's. In this paper we will
make use of the work of Beltrametti, Sommese and Wi\'sniewski
(\cite{BSW}), where they study polarized manifolds $(X,H)$ swept out
by lines, i.e. rational curves of $H$-degree one.

A naturally related goal is the classification of projective varieties
$X\subset\P^N$ dominated by families of linear subspaces $L=\P^t$, see
for instance \cite{Ein}, \cite{BSW}, \cite{ABW}. The general
philosophy here is that these varieties may be classified if the
codimension $\dim(X)-t$ is small. In fact, Sato classified them for
$\dim(X)\leq 2t$ (cf. \cite[Main Thm.]{sato}) and recently Novelli and
Occhetta have completed the case $\dim(X)=2t+1$ (cf.  \cite[Thm
1.1]{nov-occ}).

One could also study varieties swept out by other types of
subvarieties.  We would like to point out two different directions. On
one side we have the {\it extendability problem}, i.e. study which
algebraic varieties may appear as an ample divisor on a smooth
variety. For instance, it is well known that the quadric $\Q^{n}$ can
only appear as an ample divisor in $\P^{n+1}$ or $\Q^{n+1}$ and that
the grassmannian of lines $\G(1,r)$ is not extendable for $r\geq 4$
(cf.  \cite{F2}).  Extendability has been studied for many other
varieties, see for instance \cite{B} and the references therein. On
the other side one could consider subvarieties of codimension bigger
than one.  The case of quadrics has been treated by several authors,
see for instance \cite{KS}, \cite{Fu} and \cite{BI}. In the three cases
they study embedded projective varieties $X\subset\P^N$ swept out by
quadrics of small codimension.

Putting together the previous considerations we find of interest the
problem of classificating varieties swept out by codimension two
grassmannians of lines. Our main result is the following:

\begin{thm}\label{thm:mainthm}
  Let $(X,H)$ be a polarized variety of dimension $2r$, $r \geq 4$.
  Suppose that $X$ is dominated by deformations of a subvariety
  $G\subset X$ isomorphic to $\mathbb{G}(1,r)$, such that $H|_G$ is
  the ample generator of $\Pic(G)$. Assume further that $H$ is very
  ample and $H^1(X, \mathcal{I}_{G/X}(H))=0$. Then either:
\begin{itemize}
\item[\ref{thm:mainthm}.1.] there exists a morphism $\Phi:X \to Y$
  onto a normal surface such that the general fiber is isomorphic to
  $G$, or
\item[\ref{thm:mainthm}.2.] $X=\mathbb{G}(1,r+1)$ and $H$ is the ample
  generator of $\Pic(X)$.
\end{itemize}
\end{thm}

For the sake of completeness we have also dealt with the cases $r=2$
and $r=3$ which are special since $\mathbb{G}(1,2)$ is linear and
$\mathbb{G}(1,3)$ is a quadric. In fact, our methods allow us to
classify $n$-dimensional polarized varieties $(X,H)$ swept out by
codimension two linear spaces or quadrics (see Propositions
\ref{prop:r=s=2} and \ref{prop:r=s=3}). Observe that the very
ampleness of the polarization is not needed in our approach, whereas
it was necessary in the results of Sato, \cite{sato}, Kachi-Sato,
\cite{KS}, and Beltrametti-Ionescu, \cite{BI}, quoted above.

The structure of the paper is the following. In Section
\ref{sec:resultados} we expose some background material, including a
result by Beltrametti, Sommese and Wi\'sniewski on the nef value
morphism of polarized varieties swept out by lines, that will be the
starting point of our classification. In Section \ref{sec:grass} we
obtain a structure result on polarized varieties $(X,H)$ swept out by
grassmannians of lines, based on their nef value morphisms. We also
study the normal bundle to those grassmannians in $X$. Section
\ref{sec:linquad} deals with the classification of polarized varieties
swept out by codimension two linear spaces and quadrics. In the case
of quadrics the problem of finding out which Del Pezzo varieties
contain quadrics appears. Our solution goes through computing the
possible normal bundles to quadrics embedded in certain weighted
projective spaces.  Finally, we finish the proof of Theorem
\ref{thm:mainthm} in Section \ref{sec:proof}. Note that this is the
only place where we need very ampleness of the polarization. The proof
involves a study of the normal bundle in $X$ to a linear subspace of
$G$ of maximal dimension, as well as the result by Novelli and
Occhetta cited above. With this ingredients at hand we study the
variety of tangents to lines in $X$ passing through a general point,
and the proof boils down to using Sato's Theorem \cite[Main
Thm.]{sato}.

\medskip

\noindent{\bf Acknowlegements:} 
We would like to thank Miles Reid for his valuable suggestions
concerning weighted projective spaces.

\subsection{Conventions and definitions.}\label{notation:nefvalue}
\medskip

We will work over the complex numbers and we will freely use the
notation and conventions appearing in \cite{hartshorne}. When there is
no ambiguity we will denote a line bundle $\cO_X(M)$ on a variety $X$
by $\cO(M)$.

Along the paper a polarized variety will be a pair $(X,H)$ where $X$
is a smooth irreducible projective variety and $H$ is an ample line
bundle on $X$. The {\it nef value} of $(X,H)$ is the minimum number
$\tau$ such that $K_X+\tau H$ is nef but not ample. If $K_X$ is not
nef then $\tau$ is rational and the $\Q$-divisor $K_X+\tau H$ is
semiample. We will denote by $\Phi:X\to Y$ the morphism with connected
fibers determined by $m(K_X+\tau H)$, $m>>0$, and we will call it the
{\it nef value morphism} of $(X,H)$.

We will denote by $(\mathbb{G}(k,n),\cO(1))$ the grassmannian of
linear subspaces of dimension $k$ in $\P^n$ polarized by the ample
generator of its Picard group, and by $(\mathbb{Q}^n, \mathcal{O}(1))$
the smooth quadric of dimension $n$ polarized by the very ample
divisor defining the embedding $\mathbb{Q}^n \subset \P^{n+1}$ as a
hypersurface of degree two.

We will say that $(X,H)$ is a {\it scroll} over a smooth projective
variety $B$ if there exists a vector bundle $E$ on $B$ such that
$X=\mathbb{P}(E)$ and $H$ is the tautological line bundle.

Given an irreducible family $\mathcal{C}$ of rational curves in $X$ we
will say that $X$ is {\it rationally chain connected by the family}
$\mathcal{C}$ if two general points of $X$ can be connected by a chain
of curves of $\mathcal{C}$. We refer to \cite{hwang-survey} and
\cite{keso} for notation and generalities on rational curves and the
variety of minimal rational tangents.

A vector bundle on a projective variety $X$ is called {\it generically
  globally generated} (g.g.g. for brevity) if it is globally generated
at the general point.


\section{Preliminars}\label{sec:resultados}

We begin by recalling the following well known features of families of
subschemes:

\begin{rem}\label{rmk:swpt} {\rm Let $(X,H)$ be a
    polarized variety. Set $G \subset X$ an irreducible smooth
    subvariety. The universal family parametrized by an irreducible
    component $\cH$ containing $[G]$ of the Hilbert scheme $\Hilb(X)$
    dominates $X$ if and only if the normal bundle of a general
    deformation of $G$ in $\mathcal{H}$ is g.g.g..  For simplicity we
    will say that $\cH$ {\it dominates} $X$, or that $X$ {\it is dominated by a
    family of deformations of} $G$.}
\end{rem}

Given a family $\cH$ of smooth subschemes of $X$, one may wish to
study semicontinuity properties on the normal bundles. In order to do
that we introduce the following notation:

\begin{notn}\label{notation:universalfamily} {\rm Let $(X,H)$ be a
    polarized variety and $[G]$ be a smooth point of a dominating
    component $\cH\subset\Hilb(X)$. Let $I=\{(p,[G']): p \in G',\;
    [G'] \in \cH\}$ be the universal family, $\pi_1$ and $\pi_2$ the
    corresponding projections, $\cH_0$ the open set of smooth points
    of $\cH$ and $I_0=\pi_2^{-1}(\mathcal{H}_0)$. Shrinking $\cH_0$ if
    necessary, we get a diagram of sheaves over $I_0$ with exact rows:
    $$
    \xymatrix@C=2.5pc{0\ar[r]&T_{I_0|\mathcal{H}_0}\ar@{=}[d]\ar[r]&T_{I_0}
      \ar[d]^{d\pi_1}\ar[r]&\pi_2^*T_{\cH_0}\ar[d]\ar[r]&0\\
      0\ar[r]&T_{I_0|\mathcal{H}_0}\ar[r]^{\overline{\pi}_1}&
      \pi_1^*T_X\ar[r]&N_{\cH}\ar[r]&0}
    $$
    where $N_{\cH}:=\coker(\overline{\pi}_1)$ verifies
    $N_{\cH}|_{G'}\cong N_{G'/X}$, for $[G']\in\cH_0$.}
\end{notn}

Varieties swept out by lines have been extensively studied. We will
make use of a particular result in this direction, extracted from a
more detailed exposition due to Beltrametti, Sommese and Wi\'sniewski
(cf.  \cite{BSW}).

\begin{thm}[\cite{BSW}~Thms.~2.1-2.5]\label{thm:BSW}
  Let $(X,H)$ be a polarized variety such that for each point $x \in
  X$ there exists a rational curve $\ell \subset X$ with $x\in \ell$
  and $H\cdot \ell=1$. With the notation of \ref{notation:nefvalue} we
  get:
\begin{itemize}
\item[\ref{thm:BSW}.1.] Either $\Phi$ contracts $\ell$ (equivalently,
  $\tau=-K_X\cdot \ell$), or $-K_X\cdot \ell+1\leq\tau\leq n+K_X\cdot
  \ell+2$ and, in particular, $-K_X\cdot \ell\leq(n+1)/2$.
\item[\ref{thm:BSW}.2.] If $-K_X\cdot \ell\geq (n+1)/2$ then
  $-K_X\cdot\ell=(n+1)/2$ unless $\Phi$ is the contraction of the
  extremal ray $\R_+[\ell]$.
\end{itemize}
\end{thm}

As an application we get the following lemma.

\begin{lemma}\label{lem:picard}
  Let $(X,H)$ be a polarized variety. Assume that $X$ is dominated by
  deformations of a smooth subvariety $G\subset X$. Assume further
  that $G$ is rationally connected by a family $\mathcal{C}$ of
  rational curves of $H$-degree one. Set $c:=\det(N_{G/X})\cdot \ell$
  for $[\ell]\in \mathcal{C}$. If $c>K_G \cdot \ell+(n+1)/2$
  then, with the notation of \ref{notation:nefvalue}, $\tau=-K_G \cdot
  \ell+c$, $\Phi$ is the contraction of the extremal ray $\R_+[\ell]$
  and $\Phi(G)$ is a point.
\end{lemma}

\begin{proof} Being $f:\P^1\to\ell$ the normalization morphism,
  the hypotheses imply that $f^*N_{G/X}$ is g.g.g. and hence it is
  nef. But $G$ is dominated by $\mathcal{C}$, hence $f^*T_G$ is nef.
  It follows that $f^*T_X$ is nef too and, equivalently, $X$ is swept
  out by rational curves of $H$-degree one. Since
  $-K_X\cdot\ell=-K_{G}\cdot\ell+c>(n+1)/2$, Theorem \ref{thm:BSW}
  implies that $\tau$ equals $-K_{G}\cdot\ell+c$ and $\Phi:X\to Y$ is
  the (fiber-type) contraction of the extremal ray $\R_+[\ell]$.
  Finally, since $G$ is rationally chain connected by the family
  $\mathcal{C}$, its image by $\Phi$ is a point.
\end{proof}


\section{Varieties swept out by grassmannians}\label{sec:grass}

Let us start this section by fixing the setup:

\begin{set}\label{setup}{\rm Let $(X,H)$ be a polarized variety of dimension
    $n=2r$. We assume that $X$ is dominated by a family of
    deformations of $G\cong \mathbb{G}(1,r)$, $r \geq 2$. Assume
    further that $\cO(1)=H|_G$ generates $\Pic(G)$ and write
    $\det(N_{G/X})\cong cH|_G$, $c\in\Z$.  }
\end{set}

\begin{rem}\label{rem:isodef}
  {\rm Note that the vanishing $H^1(\G(1,r),T_{\G(1,r)})=0$ (obtained
    by Littlewood-Richardson formula, for instance) implies that the
    general deformation of $G$ inside $X$ is isomorphic to $\G(1,r)$.}
\end{rem}

We begin by applying the results of the previous section to a
polarized variety $(X,H)$ verifying the hypotheses we have just
imposed.

\begin{prop}\label{prop:struc} Let $(X,H)$ be as in \ref{setup}. Then either
\begin{itemize}
\item[\rm \ref{prop:struc}.1.] $Y$ is a normal surface, the general
  fiber of $\Phi$ is isomorphic to $G$ and $N_{G/X}\cong \cO^{\oplus
    2}$, or
\item[\rm \ref{prop:struc}.2.] $Y$ is a smooth curve, the general
  fiber of $\Phi$ is either $\mathbb{P}^3$ or a smooth $5$-dimensional
  quadric and $N_{G/X}\cong \cO\oplus \cO(1)$, or
\item[\rm \ref{prop:struc}.3.] $Y$ is a smooth curve, the general
  fiber of $\Phi$ is $\P^5$ and $N_{G/X}\cong\cO\oplus \cO(2)$, or
\item[\rm \ref{prop:struc}.4.] $Y$ is a point, {\rm Pic}$(X)=\Z H$ and
  $-K_X=(r+1+c)H$.
\end{itemize}
\end{prop}

\begin{proof} Since $N_{G/X}$ is g.g.g.,
  then $c \geq 0$. Hence, by Lemma \ref{lem:picard}, $\tau=r+1+c$ and
  $\Phi$ contracts $G$ to a point, which in particular gives $\dim(Y)
  \leq 2$.
  
  If $\dim(Y)=2$, since the fibers of $\Phi$ are connected, the
  general deformation of $G$ coincides with the fiber containing it,
  hence \ref{prop:struc}.1 holds.
  
  If $Y$ is a point, then a multiple of $K_X+\tau H$ is trivial.
  Since $\Phi$ is an elementary contraction, it follows that $X$ is a
  Fano manifold of Picard number $1$. In particular $\Pic(X)$ has no
  torsion, thus $K_X+\tau H$ is trivial too and \ref{prop:struc}.4
  follows.
  
  Thus we are left with the case $\dim(Y)=1$. Let us denote by $F$ the
  general fiber of $\Phi$, that contains a grassmannian $G$.  Applying
  Lemma \ref{lem:picard} to $(F,H|_F)$ and using that the nef value
  morphism $\Phi_F$ coincides with $\Phi|_F$ we obtain that $F$ is a
  Fano manifold whose Picard group is generated by $H|_F$. But $G$
  appears as an effective, and hence ample, divisor on $F$. In
  particular $c\geq 1$. On the other side it is classically known that
  this is only possible (cf.  \cite[Theorem~5.2]{F1}) if $r=2,3$. It
  follows that $G$ is either $\mathbb{P}^2$ or a smooth quadric of
  dimension $4$. If the former holds then $F=\mathbb{P}^3$, the
  exact sequence
\begin{equation}\label{normaltothefiber}
0 \to N_{G/F} \to N_{G/X} \to \mathcal{O} \to
0,\end{equation} splits and we get
\ref{prop:struc}.2. If the later holds then $-K_G=4H|_G$. This
implies that $-K_F=(4+c)H|_F$, and applying Kobayashi-Ochiai
characterization of quadrics and projective spaces, \cite{KO},
either $c=1$ and $F$ is a $5$-dimensional quadric, or $c=2$ and
$F$ is isomorphic to $\P^5$. The splitting of the exact sequence
(\ref{normaltothefiber}) concludes the proof.
\end{proof}

\begin{rmk}\label{rmk:hyperqudricfibrations} {\rm Let us remark that
    in the case \ref{prop:struc}.1 when $r=3$ we get smoothness of
    $Y$ when $H$ is very ample or $\Phi$ is equidimensional, see
    \cite[Thm. B]{ABW2} and \cite[(2.3)]{BS}. In fact, $Y$ is
    conjectured to be smooth, see \cite[Conj.  14.2.10]{BSbook}. In
    \ref{prop:struc}.3 or in the first case of \ref{prop:struc}.2, if
    $H$ is very ample, all fibers of $\Phi$ are isomorphic. However
    this is not true in general, as one can see by considering certain
    quadric sections of the Segre embedding of $\P^1 \times \P^5
    \subset \P^{11}$.}
\end{rmk}

At this point, one would like to determine $N_{G/X}$ also in the case
\ref{prop:struc}.4, but this task is not as simple as in the other
cases.  In fact certain restrictions might be imposed in order to get
our classification. However the following lemmas allow us to claim
that $c$ is different from zero.

\begin{lemma}\label{lemma:c=0}  Let $(X,H)$ be as in
  \ref{setup}. If $c=0$ then $N_{G/X}$ is trivial.
\end{lemma}

\begin{proof} Since $N_{G/X}$ is g.g.g.,
  then $\dim(H^0(X,N_{G/X})) \geq 2$. Taking two independent sections
  we get a generically injective morphism $\mathcal{O}^{\oplus 2} \to
  N_{G/X}$. This produces a nonzero global section of the
  $\det(N_{G/X})$. Since $c=0$, this gives $\det(N_{G/X})\cong\cO$ and
  $N_{G/X}\cong\mathcal{O}^{\oplus 2}$.
\end{proof}

\begin{lemma}\label{lemma:trivialnormal} Let $X$ be an irreducible
  smooth projective variety of dimension bigger than or equal to $4$.
  Assume that $X$ contains a codimension two smooth subvariety $G
  \subset X$ such that $b_2(G)=1$ and $H^1(G,\cO)=0$. If
  $N_{G/X}\cong\cO^{\oplus 2}$ then $\rho(X)>1$.
\end{lemma}

\begin{proof} The hypotheses imply that
  $\dim(H^0(G,N_{G/X}))=2$ and $H^1(G,N_{G/X})\\=0$, hence $[G]$ is a
  smooth point of a $2$-dimensional component $\cH\subset\Hilb(X)$.
  Furthermore, with the notation introduced in
  \ref{notation:universalfamily} we may assume that there exists an
  open subset $\cH_0 \subset \cH$ such that any element $[G_t] \in
  \mathcal{H}_0$ corresponds to a smooth projective subvariety $G_{t}
  \subset X$ for which $N_{G_{t}/X}\cong\mathcal{O}^{\oplus
    2}$, $d\pi_1$ is an isomorphism and $b_2(G_t)=1$.  This provides a
  finite morphism $f=\pi_1|_{I_0}:\mathcal{I}_0 \to X_0$ from the
  universal family over $\mathcal{H}_0$ onto an open subset $X_0
  \subset X$.
  
  If $f$ is generically one to one, then the Picard number of $X$
  cannot be one. Thus we may assume that $\deg(f)>1$.
  
  Take a general $[G_t] \in \mathcal{H}_0$ and define: $C_t=\{[G_u]
  \in \mathcal{H}_0:\; G_u \cap G_t \ne \emptyset \}$. The subscheme
  $C_t\subset\cH_0$ is nonempty by the hypothesis on $\deg(f)$ and it
  is different from $\cH_0$ since $f$ is a local isomorphism at every
  point. Given $[G_u] \in C_t$ we claim that $C_t \subseteq C_u$. In
  fact since $c_2(N_{G_t/X})=0$, the self intersection formula tells
  us that $G_u\cap G_t$ is a divisor on $G_t\cong \G(1,r)$. Hence,
  given any $[G_v]\in C_t$ we get $G_v\cap G_u\neq\emptyset$. The same
  argument leads to the equality $C_t=C_u$.
  
  As an abuse of notation let $C_t$ stand now for the union of the one
  dimensional components of $C_t$. Now define $D_0 =\cup_{u \in C_t}
  G_{u}$ which is a divisor on $X_0$ and observe that for the general
  $[G_s] \in \mathcal{H}_0$ we have $D_0 \cap G_{s} =\emptyset$. If
  $\rho(X)=1$ is one then the closure $D$ of $D_0$ in $X$ would meet
  $G_s$ so that $(D \setminus D_0) \cap G_s \neq \emptyset$ for any
  $[G_s] \in \mathcal{H}_0$. But $D \setminus D_0$ is a finite union
  of codimension two subvarieties of $X$, any of them algebraically
  equivalent to $G$.  Then, the general $G_s$ would meet an
  irreducible component of $D \setminus D_0$ in a divisor of $G_s$,
  contradicting the fact that $G_s \cap G_{s'}= \emptyset$ for general
  $s,s' \in \cH_0$.
\end{proof}


\section{Varieties swept out by codimension two linear spaces and quadrics}\label{sec:linquad}

The main result of this paper, Theorem \ref{thm:mainthm}, classifies,
under certain assumptions, varieties swept out by deformations of
$\mathbb{G}(1,r)$, with $r\geq 4$.  For the sake of completeness, we
have addressed in this section the cases $r=2,3$, for which some ad hoc
arguments are needed. They allow us to classify
$n$-dimensional varieties swept out by codimension two linear spaces
and quadrics. These problems have been already addressed by many
authors, see for instance \cite{sato}, \cite{KS} and \cite{BI}. Note
that they allow higher codimension but they assume very ampleness of
the polarization, which is not necessary in our case.

An analogue of Proposition \ref{prop:struc} already allow us to study
varieties swept out by codimension $2$ linear spaces.  We have skipped
the proof since it follows verbatim \ref{prop:struc}. Note that we
need to use that projective spaces are rigid
($H^1(\P^{n-2},T_{\P^{n-2}})=0$) and that the only varieties
containing linear spaces as ample divisors are linear spaces
themselves. Note also that the only quadric containing codimension $2$
linear spaces is $\Q^4$.

\begin{prop}\label{prop:r=s=2} Let $(X,H)$ be a polarized variety 
  of dimension $n \geq 4$.  Suppose that $X$ is dominated by a family
  of deformations of $L\cong\P^{n-2}$, with $H|_L\cong\cO(1)$. Then,
  with the notation of \ref{notation:nefvalue}, either:
\begin{itemize}
\item[\rm \ref{prop:r=s=2}.1.] $Y$ is a normal surface and the general
  fiber of $\Phi$ is $\P^{n-2}$, or
\item[\rm \ref{prop:r=s=2}.2.] $Y$ is a smooth curve and the general
  fiber of $\Phi$ is $\P^{n-1}$, or
\item[\rm \ref{prop:r=s=2}.3.] $(X,H)=(\P^n, \mathcal{O}(1))$, or
\item[\rm \ref{prop:r=s=2}.4.] $(X,H)=(\Q^4,\cO(1))$.
\end{itemize}
\end{prop}

\begin{rmk}\label{rmk:n=3} {\rm The hypothesis $n \geq 4$ in
    \ref{prop:r=s=2} is needed in order to get the bound $K_L \cdot
    \ell+(n+1)/2<0$ that allows us to apply Lemma \ref{lem:picard}.
    If $n=3$ and $c=0$ these arguments do not work. Nevertheless we
    can apply basic results of adjunction theory, see \cite[Section
    1]{I}, to describe this case.  If $(X,H)$ is not
    $(\P^3,\mathcal{O}(1))$, $(\mathbb{Q}^3, \mathcal{O})$ or a scroll
    over a curve, then $K_X+2H$ is nef and so $\tau=2$. Hence, either
    $(X,H)$ is a Del Pezzo threefold, or a quadric fibration over a
    smooth curve, or a scroll over a surface.}
\end{rmk}

In the case of quadrics, reasoning as above we obtain the following:

\begin{prop}\label{prop:r=s=3} Let $(X,H)$ be a polarized variety of 
  dimension $n \geq 6$.  Suppose that $X$ is dominated by a family of
  deformations of $L\cong\Q^{n-2}$, with $H|_L\cong\cO(1)$. Then,
  with the notation of \ref{notation:nefvalue}, either:
\begin{itemize}
\item[\rm \ref{prop:r=s=2}.1.] $Y$ is a normal surface and the general
  fiber of $\Phi$ is $\Q^{n-2}$, or
\item[\rm \ref{prop:r=s=2}.2.] $Y$ is a smooth curve and the general
  fiber of $\Phi$ is either $\P^{n-1}$ or $\Q^{n-1}$, or
\item[\rm \ref{prop:r=s=2}.3.] $(X,H)=(\P^n,\cO(1))$, or
\item[\rm \ref{prop:r=s=2}.4.] $(X,H)=(\Q^n, \mathcal{O}(1))$, or
\item[\rm \ref{prop:r=s=2}.5.] $X$ is a Del Pezzo variety of Picard
  number $1$ and $H$ is the ample generator of $\Pic(X)$.
\end{itemize}
\end{prop}

The classification will be completed by determining which Del Pezzo
varieties may be swept out by codimension two quadrics:

\begin{prop}\label{prop:delpezzo}
  Let $X$ be a Del Pezzo variety of dimension $n \geq 4$ and
  $\Pic(X)=\Z H$. If $X$ contains a $(n-2)$-dimensional smooth quadric
  $\mathbb{Q}^{n-2}$ of $H$-degree $2$, then $X$ is isomorphic to a linear
  section of $\mathbb{G}(1,4)$.
\end{prop}

This fact is based on Fujita's classification of Del Pezzo varieties
(cf. \cite[8.11, p. 72]{F2}). We are interested in those of Picard
number $1$, which are:
\begin{itemize}
\item[I.] $X\cong X_3\subset\P^{n+1}$ is a hypersurface of degree
  three, or
\item[II.] $X\cong X_{2,2}\subset\P^{n+2}$ is the complete
  intersection of two quadrics, or
\item[III.] $X\cong X_4$ is a degree four hypersurface in the weighted
  projective space $\P(1^{n+1},2)$, or
\item[IV.] $X\cong X_6\subset\P(1^n,2,3)$ is a degree 6 hypersurface,
  or
\item[V.] $X$ is isomorphic to a linear section of $\G(1,4) \subset
  \P^9$.
\end{itemize}

\begin{proof}[Proof of Proposition~\ref{prop:delpezzo}] For each case 
  denote by $\P$ the corresponding ambient space. We will discard
  Types I to IV by showing that $\Q:=\mathbb{Q}^{n-2}$ does not meet
  the singular locus of $\P$ and that the normal bundle $N_{\Q/\P}$ of
  a quadric $\Q$ in $\P$ of $H$-degree $2$ does not admit a surjective
  morphism onto $N_{X/\P}|_{\Q}$. More concretely, we claim that the
  pair $\big(N_{\Q/\P},N_{X/\P}|_{\Q}\big)$ takes the values
  $\big(\cO(H)^2\oplus\cO(2H),\cO(3H)\big)$,
  $\big(\cO(H)^3\oplus\cO(2H),\cO(2H)^2\big)$,
  $\big(\cO(H)\oplus\cO(2H)^2,\cO(4H)\big)$,
  $\big(\cO(2H)^2\oplus\cO(3H),\cO(6H)\big)$ for Types I to IV,
  respectively. In Types I and II the line bundle $H$ is very ample
  and the statement is immediate. We will show how to discard Type IV,
  being Type III completely analogous. The following argument was
  suggested to us by M. Reid.
  
  An embedding $\Q\subset\P(1^n,2,3)=\P$ is given by sections
  $s_1,\dots,s_n\in H^0(\Q,\cO(1))$, $t\in H^0(\Q,\cO(2))$ and $u\in
  H^0(\Q,\cO(3))$.
  
  If $s_1,\dots,s_n$ are linearly independent then they generate the
  homogeneous coordinate ring of $\Q$. Choosing appropriate
  weighted homogeneous coordinates $x_1,\dots,x_n,y,z$ in $\P$ we may
  assume that $\Q\subset\P$ is a complete intersection defined by the
  following equations:
\begin{equation}\label{eq:quadric}
(x_1^2=x^2_2+\dots +x_n^2,\; y=x_1f_1+f_2,\; z=x_1g_2+g_3),
\end{equation}
being $f_i$ and $g_i$ homogeneous polynomials of degree $i$ in
$x_2,\dots,x_n$.  This implies that the normal bundle takes the
desired form.

If $s_1,\dots,s_n$ are not linearly independent then we may assume
that $\Q$ lies on a subvariety of equations $s_1=\dots=s_i=0$,
isomorphic to $\P(1^{n-i},2,3)$. This case may be ruled out by showing
that the quadric defined by the equations (\ref{eq:quadric}) cannot be
projected isomorphically already into $\P(1^{n-1},y,z)$ (eliminating
the variable $x_1$). In fact the image of $\Q$ by this projection is
defined by equations
$$
\rank\left(\begin{array}{cccc}y-f_2&z-g_3&(x^2_2+\dots+
    x^2_n)f_1&(x^2_2+ \dots+
    x^2_n)g_2\\f_1&g_2&y-f_2&z-g_3\end{array}\right)\leq 1,
$$
and the set of points having positive dimensional inverse image,
defined by equations $f_1=g_2=y-f_2=z-g_3=0$, is nonempty.
\end{proof}

\begin{rmk}\label{rmk:otherpezzo}
  {\rm Note that if we skip the hypothesis $\Pic(X)=\Z$ in the previous
  proposition, there is just another possibility, namely
  $X\cong\P^2\times\P^2$.}
\end{rmk}

\begin{rmk}\label{rmk:n=4,5}{\rm Similarly to \ref{rmk:n=3}, let us point
    out that the hypothesis $n \geq 6$ of \ref{prop:r=s=3} is needed in order
    to apply Lemma \ref{lem:picard}. We observe that if $n=5$ and $c
    \ne 0$ then \ref{lem:picard} applies and the same conclusion as in
    \ref{prop:r=s=3} follows. If $n=5$, $c=0$ then, by \cite[Thm.
    2.5]{BSW}, either $X=\P^2 \times \P^3$, or $\tau=3$, and
    $\Phi:X \to Y$ contracts $\mathbb{Q}^3$. If $\dim(Y)=2$ then $X$
    is as in \ref{prop:r=s=3}.1. If $\dim(Y)=1$ then, for the general
    fiber $F$, we get $-K_F \cdot \ell=3$ so that $\Phi_F=\Phi|_F$ is
    the contraction of a extremal ray. Hence $X$ is as in
    \ref{prop:r=s=3}.2. If $\dim(Y)=0$ then $\rho(X)>1$ by
    \ref{lemma:trivialnormal}. Hence $\Phi$ is not the contraction of
    a extremal ray and \cite[2.5.3]{BSW} together with \cite{Wis}
    describe $(X,H)$ precisely.
    
    If $n=4$ the situation is slightly different since $\mathbb{Q}^2$
    contains two different families of lines. Nevertheless adjunction
    theory arguments (cf.  \cite[Section 1]{I}, \cite{BSbook}) and the
    understanding of the nef morphism of $(X,H)$ and of its first
    reduction (cf.  \cite{BSW}) allow us to give a more explicit
    description of $(X,H)$. In fact if $X$ is not $\P^4$, $\Q^4$ or a
    scroll over a curve, then $K_X+3H$ is nef and in particular $\tau
    \leq 3$ and $c\leq 1$.  Now, with the exception of the cases in
    which $(X,H)$ is either Del Pezzo (see Remark
    \ref{rmk:otherpezzo}), or a quadric fibration onto a curve, or a
    scroll over a surface, we may take the first reduction $(X',H')$,
    that verifies that $K_{X'}+3H'$ is ample. In particular $\tau
    <3$.  Now, by \cite[2.1]{BSW} and what we have proved, $2 \leq
    \tau <3$.  Moreover, $\tau=2$ by \cite[Thm. 7.3.4]{BSbook}.
    Hence, the nef value morphism $\Phi:X' \to Y'$ of $(X',L')$
    contracts $\mathbb{Q}^2$. If $\dim(Y')=2$ then the general fiber
    is $\Q^2$.  If $\dim(Y')=1$ then the general fiber $F$ has 
    $\rho(F)>1$ and $F$ is one
    of the list of \cite{Wis}. If $\dim(Y)=0$ then $X'$ is a Fano
    variety of index two, classically called {\it Mukai varieties},
    described in \cite{clm} and \cite{mukai1},
    \cite{mukai2}.  }
\end{rmk}


\section{Proof of the main theorem}\label{sec:proof}
We are ready to prove Theorem \ref{thm:mainthm}. In view of
Proposition \ref{prop:struc} and Lemma \ref{lemma:trivialnormal} we
may assume that $\Pic(X)=\Z H$ and that $\det(N_{G/X})=cH|_G$ with
$c>0$. Since we are assuming that $H$ is very ample, we will consider
$X$ as a subvariety of $\P^N:=\P(H^0(X,H))$ and study linear
subvarieties of $\P^N$ contained in $X$.

In fact our proof involves describing the normal bundle $N_{L/X}$ of a
general $(r-1)$-dimensional linear subspace $L\subset G$. Note that
$\rank(N_{L/X})=\dim(L)+2$, hence, even if we check that $N_{L/X}$ is
uniform, we cannot infer that it is homogeneous. In fact, it has been
conjectured that uniform vector bundles on $\P^s$ of rank smaller than
$2s$ are homogeneous (cf. \cite{BE}), and homogeneous vector bundles
on $\P^s$ are classified for rank smaller than or equal to $s+2$ (cf.
\cite[3.4, p.~70]{OSS}, \cite{ellia}). However the conjecture has been
confirmed only for rank smaller than or equal to $s+1$, and some extra
cases for $s$ small. Nevertheless, in our particular case, we may
prove the following:

\begin{lemma}\label{lemma:normal} In the conditions of 
  Theorem \ref{thm:mainthm}, assume further that $\rho(X)=1$, and let
  $L\subset G$ be a general $(r-1)$-dimensional linear subspace. Then
  either
\begin{itemize}
\item[\ref{lemma:normal}.1.]  $c=1$ and
  $N_{L/X}=T_L(-1)\oplus\mathcal{O}(1) \oplus \mathcal{O}$, or
\item[\ref{lemma:normal}.3.] $c=2$ and
  $N_{L/X}=T_L(-1)\oplus\mathcal{O}(1)^{\oplus 2}$.
\end{itemize}
\end{lemma}

\begin{proof} Take a general line $\ell\subset L\subset G$. 
  Since $N_{G/X}$ is g.g.g., then $N_{G/X}|_{\ell}=\mathcal{O}(a_\ell)
  \oplus \mathcal{O}(b_\ell)$ with $0 \leq a_\ell \leq b_\ell$ and
  $c=a_\ell+b_\ell$.  The vanishing $h^1(X, \mathcal{I}_{G/X}(H))=0$
  implies that the ideal sheaf $\mathcal{I}_{G/\P^N}$ is generated by
  quadrics and, in particular, $N_{G/\P^N}^*(2)$ and its quotient
  $N_{G/X}^*(2)$ are globally generated.  It follows that $a_\ell \leq
  b_\ell \leq 2$ and $c \leq 4$. Moreover \cite[Lemma~2.4]{sato} tells
  us that the restriction of $N_{G/X}|_L$ to any line $\ell'\subset L$
  through a general point of $L$ has the same splitting type,
  $(a_\ell,b_\ell)$.
  
  If $c=1$, the above description implies that $a_\ell=0$, $b_\ell=1$
  for any line $\ell\subset L$ passing through a general point $x\in
  L$.  Therefore $N_{G/X}|_L=\mathcal{O}(1) \oplus \mathcal{O}$ by
  \cite[Thm.  1.1]{sato}, and the exact sequence
\begin{equation}\label{exactsequence}
0 \to N_{L/G}\cong T_{L}(-1) \to N_{L/X} \to N_{G/X}|_{L} \to
0.
\end{equation}
splits: in fact $H^1(L,T_L(s))=0$ for all $s$ since $\dim(L)\geq 3$.
This leads us to the case \ref{lemma:normal}.1.

The same argument applies to $c=3$ and $c=4$. But in both cases we get
that $\cO(2)$ is a direct summand of $N_{L/X}$, contradicting the fact
that this is a subsheaf of $N_{L/\P^N}\cong\cO(1)^{\oplus n-r+1}$.

It remains to deal with the case $c=2$. Let us observe that in this
case $N_{G/X}\cong N_{G/X}^*(2)$ is globally generated, hence in
particular $N_{G/X}|_L$ is nef and Griffiths vanishing theorem
\cite[Variant~7.3.2]{lazarsfeld} tells us that
$H^{i}(L,N_{G/X}|_L(-2))=0$ for $r>4$, $i>0$ and
$H^{i}(L,N_{G/X}|_L(-3))=0$ for $r>5$, $i>0$.  In particular taking
cohomology on the Euler sequence tensored with $N^*_{G/X}|_L$ and
using the isomorphism $N_{G/X}|_L\cong N^*_{G/X}|_L(2)$, we obtain
$H^1(L,N^*_{G/X}|_L \otimes T_L(-1))=0$ and the exact sequence
(\ref{exactsequence}) splits for $r>5$.  Now observe that $N_{L/X}$ is
nef (as an extension of two nef vector bundles) and injects into
$N_{L/\P^N}$, thus it is uniform and its splitting type is composed of
0's and 1's. The splitting of (\ref{exactsequence}) implies that
$N_{G/X}|_L$ is uniform too and we get \ref{lemma:normal}.3 for $r>5$
by \cite[Thm.~3.2.3]{OSS}.

If $r=5$ and $c=2$, consider the tautological line bundle $\xi$ of the
projective bundle $\pi:\P(N_{G/X}|_L)\to L$. Since $N_{G/X}|_L$ is nef
then the Chern-Wu relation implies $\xi^4\cdot
\pi^*c_1(\mathcal{O}(1))=8-4c_2 \geq 0$, where $c_2$ stands for the
degree of the second Chern class of $N_{G/X}|_L$. Thus $c_2 \leq 2$.
But nefness also implies that $\xi^5=16-12c_2+c_2^2 \geq 0$, hence
$c_2\leq 1$, $\xi^5>0$ and $N_{G/X}|_L$ is big. In particular
\cite[Variant 7.3.3]{lazarsfeld} leads us to the vanishing
$h^{i}(L,N_{G/X}|_L(-a))=0$ for $i>0$ and $a=2,3$, allowing us to
conclude as in the case $r>5$.

The case $c=2$ and $r=4$ must be treated in a different manner. In
this case $N_{L/X}$ is a uniform vector bundle of rank $5$ and
splitting type $(0,0,1,1,1)$. Using the classification given in
\cite{BE} we get that $N_{L/X}$ is either
\begin{itemize}
\item $T_L(-1) \oplus \mathcal{O}(1)^{\oplus 2}$, or
\item $\mathcal{O}(1)^{\oplus 3} \oplus \mathcal{O}^{\oplus 2}$, or
\item $\Omega_L(2) \oplus \mathcal{O}(1) \oplus \mathcal{O}$.
\end{itemize}
The first case leads us again to \ref{lemma:normal}.3, and the last
two cases can be excluded by proving that $N_{L/X}$ cannot contain
$\cO$ as a direct summand. In fact, using the exact sequence
(\ref{exactsequence}), and taking into account that
$\Hom(T_L(-1),\cO)=0$, the cokernel $N_{G/X}|_L$ would contain $\cO$
as a direct summand, too, and the only possibility would be
$N_{G/X}|_L=\cO\oplus\cO(2)$. But then the sequence
(\ref{exactsequence}) would split, contradicting again the fact that
$N_{L/X}$ is a subsheaf of $N_{L/\P^N}\cong\cO(1)^{\oplus n-r+1}$.
\end{proof}

\begin{rmk}\label{generationbyquadrics} {\rm Let us point out that the
    hypothesis $H^1(X, \mathcal{I}_{G/X}(H))=0$ can be substituted by
    the hypothesis on the ideal sheaf $\mathcal{I}_{G/\P^N}$ to be
    generated by quadrics.}
\end{rmk}

The following arguments finish the proof of Theorem \ref{thm:mainthm}.

\begin{proof}[End of the proof] With the same notation and 
  assumptions as above, note that Lemma \ref{lemma:normal} implies
  $H^1(L,N_{L/X}(a))=0$ for all $a\in\Z$ and, in particular $[L]$ is a
  smooth point in $\Hilb(X)$. Denote by $\cH$ the unique component of
  $\Hilb(X)$ containing $[L]$ and, with the notation presented in
  \ref{notation:universalfamily} (where $L$ plays the role of $G$),
  denote by $N_\cH$ the vector bundle on the universal family $I_0$
  verifying that $N_\cH|_{L'}\cong N_{L'/X}$ for $[L']\in \cH_0$.  It
  will allow us to use semicontinuity.
  
  We claim that, if $c=1$ (respectively $c=2$), the normal bundle
  $N_{L'/X}$ splits again as $T_{L'}(-1)\oplus \mathcal{O}(1)\oplus
  \mathcal{O}$ (resp. as $T_{L'}(-1)\oplus \mathcal{O}(1)^{\oplus
    2}$). If $c=1$ (resp.  $c=2$) and $0 \leq j \leq r-1$, then
  $H^{i}(L,N_{L/X}^*(-j)) \ne 0$ if and only if $i=j=0,1, r-1$ (resp.
  $i=j=1, r-1$). Applying semicontinuity to $N_\cH^*$ and its twists,
  the same occurs for the general $L'$ so that we conclude by using
  the Beilinson spectral sequence \cite[Thm.~3.1.3]{OSS}.
  
  In any case the normal bundle of a general deformation $L'$ of $L$
  contains $\cO(1)$ as a direct summand, providing a smooth hyperplane
  section $X':=H\cap X$ containing $L'$ (cf. \cite{ABW},
  \cite[Cor.~1.7.5]{BSbook}). Therefore, by Bertini theorem, the
  general hyperplane containing $L'$ is smooth, too. Since $L'$ is
  general we may assume that such a section exists passing through the
  general point $x\in X$.  Moreover by construction of $X'$, either
  $N_{L'/X'}\cong T_{L'}(-1)\oplus \mathcal{O}$ if $c=1$, or
  $N_{L/X'}\cong T_{L'}(-1)\oplus \mathcal{O}(1)$ if $c=2$ (cf.
  \cite[Lem.~4.3]{nov-occ}). Note also that Lefschetz theorem provides
  $\Pic(X')=\Z$.
  
  At this point we apply \cite[Cor. 6.1.4]{nov-occ} to $X'$, obtaining
  that it is isomorphic to a linear section of the Pl\"ucker embedding
  of $\G(1,r+1)$ and $c$ is necessarily equal to $1$.
  
  Let $\mathcal{C}_x\subset\P(\Omega_{X,x})$ be the variety of minimal
  rational tangents to $X$ at a general point $x$ (cf.
  \cite{hwang-survey}), which in this case is the set of tangent
  directions to lines in $X$ through $x$.  Since $X'$ is a hyperplane
  section of $\mathbb{G}(1,r+1)$ then the corresponding hyperplane
  section of $\mathcal{C}_x$ is a hyperplane section of the Segre
  embedding $\P^1 \times \P^{r-1} \subset \P^{2r-1}$, in particular it
  is a variety of minimal degree in $\P(\Omega_{X',x})$.  Being
  $\mathcal{C}_x$ smooth by \cite[Prop.~1.5]{hwang-survey},
  $\mathcal{C}_x$ must be the Segre embedding $\P^1 \times \P^{r-1}
  \subset \P^{2r-1}$.
  
  In particular, through a general point $x \in X$ there exists an
  $r$-dimensional linear space $M \subset X$.  Moreover, the
  restriction of the normal bundle $N_{M/X}$ to a codimension one
  linear subspace $L' \subset M$ is $N_{M/X}|_{L'}\cong N_{L'/X'}\cong
  T_{L'}(-1)\oplus \mathcal{O}$ and we conclude that $X\cong\G(1,r+1)$
  by \cite[Main Thm.]{sato}.
\end{proof}

\bibliographystyle{amsalpha}

\end{document}